\documentclass[11pt,a4paper]{article}
\usepackage[top=3cm, bottom=3cm, left=3cm, right=3cm]{geometry}
\usepackage[latin1]{inputenc}
\usepackage{amsmath}
\usepackage{amsthm}
\usepackage{amsfonts}
\usepackage{amssymb}
\usepackage{graphics}
\usepackage[hang,flushmargin]{footmisc} %nonindent footnote

\usepackage{amsmath,amsthm,amssymb,graphicx, multicol, array}
\usepackage{enumerate}
\usepackage{enumitem}
\usepackage{xcolor}
\usepackage{currfile}

\usepackage[pagebackref,bookmarks,colorlinks,breaklinks]{hyperref}

\hypersetup{linkcolor=blue,citecolor=red,filecolor=blue,urlcolor=blue} 

\usepackage{tabto}

\numberwithin{equation}{section}

\usepackage{tikz}
\usepackage{pgfplots}

\usetikzlibrary{matrix}
\usepackage{mathrsfs}

\newtheorem{thm}{Theorem}[section]
\newtheorem{lem}{Lemma}[section]

\newtheorem{rmk}{Remark}[section]

\newcommand{\N}{\mathbb{N}}

\usepackage{fancyhdr}
\title{The Product $e \pi$ Is Irrational}
\date{}
\author{N. A. Carella}

\begin{document}
	%\doublespacing
	\thispagestyle{empty}
	\date{}

	\maketitle
	\textbf{\textit{Abstract}:} This note shows that the product $e \pi$ of the natural base $e$ and the circle number $\pi$ is an irrational number. \let\thefootnote\relax\footnote{ \today \date{} 
		}
	
	\tableofcontents

	%ssssssssssssssssssssssssssssssssssssssssssssssssssssssssssssssssssssssssssssssssssssssssssssssssssssssssssssssssssssssssssssssssssssssssssssssssssssssssssssssssssssssssssssssssssss
	%ssssssssssssssssssssssssssssssssssssssssssssssssssssssssssssssssssssssssssssssssssssssssssssssssssssssssssssssssssssssssssssssssssssssssssssssssssssssssssssssssssssssssssssssssssss
	%ssssssssssssssssssssssssssssssssssssssssssssssssssssssssssssssssssssssssssssssssssssssssssssssssssssssssssssssssssssssssssssssssssssssssssssssssssssssssssssssssssssssssssssssssssss
	
	\section{Introduction} \label{S001}
	The number $e= 2.718281828459 \ldots$ was proved to be irrational by Euler, circa 1744. The proof uses a differential equation to show that the continued fraction 
	\begin{equation}e=[2;2,1,2,1,1,4,1,1,6,1,1,8, \ldots]
	\end{equation} 
	is infinite, see \cite{EL1744},  \cite[Theorem 3.10]{SJ2005}. Later, a simpler proof based on the infinite series 
	\begin{equation}
		e=\sum_{n \geq 0}\frac{1}{n!}=1+\frac{1}{2!}+\frac{1}{3!}+\frac{1}{4!}+\cdots
	\end{equation}
	was found by Fourier in 1815. Many versions of the Fourier classical proof are known for $e^r$, where $r \in \mathbb{Q}$ is a rational number, and other numbers, see \cite{PL1953}, \cite[p.\ 35]{AZ2014}, \cite{SJ2005}. The number $\pi=3.141592653589 \ldots$ was proved to be irrational by Lambert, circa 1760, see \cite[p.\ 129]{BB2004}. The proof uses the continued fraction of the tangent function $\tan(x)$, the fact that the numbers $\tan(r)$ are irrationals for any nonzero rational number $r \in \mathbb{Q}$, and the value $\arctan(1)=\pi/4$ to indirectly show that the continued fraction 
	\begin{equation}
		\pi=[3;7,15,1,292,1,1,1,2,1,3,1,14, \ldots]
	\end{equation}  
	is infinite, see \cite{BB2004}, \cite{LM1997}, \cite{NI1947}. Later, simpler versions and new proofs were found by several authors, \cite{NI1947}, \cite[p.\ 35]{AZ2014}, \cite{SJ2005}. The above short compendium is a glimpse at the vast mathematical literature devoted to the analysis of the numbers $e$ and $\pi$.\\
	
	The arithmetic natures of the product $e \cdot \pi=8.539734222673 \ldots$, and  of the the sum $e+\pi$ are not known. In this note the known information on the continued fractions and the convergents of the two irrational numbers $e$ and $\pi$ are used here to construct an infinite subsequence of rational approximations for the product $e \pi$.
	
	\begin{thm} \label{thm1}
		The product $e \pi$ is an irrational number.
	\end{thm}

	The earlier sections cover the basic required background, and the proof of Theorem \ref{thm1} appears in Section \ref{S004}. An algorithm that illustrates the effectiveness of this result appears in Section \ref{S005}.

%222222222222222222222222222222222
\section{Foundation}
Except for Theorem \ref{thm5}, all the materials covered in this section are standard results in the literature, see \cite{HW2008}, \cite{LS1995}, \cite{NZ1991}, \cite{RH1994}, \cite{SJ2005}, et alii. \\

A real number \(\alpha \in \mathbb{R}\) is called \textit{rational} if \(\alpha = a/b\), where \(a, b \in \mathbb{Z}\) are integers. Otherwise, the number
is \textit{irrational}. The irrational numbers are further classified as \textit{algebraic} if \(\alpha\) is the root of an irreducible polynomial \(f(x) \in
\mathbb{Z}[x]\) of degree \(\deg (f)>1\), otherwise it is \textit{transcendental}.\\

\begin{thm} \label{thm2} If a real number \(\alpha \in \mathbb{R}\) is a rational number, then there exists a constant \(c = c(\alpha )\) such that
	\begin{equation}
		\frac{c}{q}\leq \left|  \alpha -\frac{p}{q} \right|
	\end{equation}
	holds for any rational fraction \(p/q \neq \alpha\). Specifically, \(c \geq  1/b\text{ if }\alpha = a/b\).
\end{thm}

This is a statement about the lack of effective or good approximations for any arbitrary rational number \(\alpha \in \mathbb{Q}\) by other rational numbers. On the other hand, irrational numbers \(\alpha \in \mathbb{R}-\mathbb{Q}\) have effective approximations by rational numbers. If the complementary inequality \(\left|  \alpha -p/q \right| <c/q\) holds for infinitely many rational approximations \(p/q\), then it already shows that the real number \(\alpha \in \mathbb{R}\) is irrational, so it is sufficient to prove the irrationality of real numbers.

\begin{thm}[Dirichlet]\label{thm3} 
	Suppose $\alpha \in \mathbb{R}$ is an irrational number. Then there exists an infinite
	sequence of rational numbers $p_n/q_n$ satisfying
	\begin{equation}
		0 < \left|  \alpha -\frac{p_n}{q_n} \right|< \frac{1}{q_n^2}
	\end{equation}
	for all integers \(n\in \mathbb{N}\).
\end{thm}

\begin{thm} \label{thm4}   Let $\alpha=[a_0,a_1,a_2, \ldots]$ be the continued fraction of a real number, and let $\{p_n/q_n: n \geq 1\}$ be the sequence of convergents. Then
	\begin{equation}
		0 < \left|  \alpha -\frac{p_n}{q_n} \right|< \frac{1}{a_{n+1}q_n^2}
	\end{equation}
	for all integers \(n\in \mathbb{N}\).
\end{thm}
This is standard in the literature, the proof appears in \cite[Theorem 171]{HW2008}, \cite[Corollary 3.7]{SJ2005}, and similar references.\\

A basic extension of the previous inequalities in Theorem \ref{thm3} and Theorem \ref{thm4} provided here uses a pair of distinct irrational numbers and the corresponding parameters.

\begin{thm} \label{thm5} 
	Let $\alpha=[a_0,a_1,a_2, \ldots ]$ and  $\beta=[b_0,b_1,b_2, \ldots ]$ be distinct continued fractions for two distinct irrational numbers $\alpha \text{ and } \beta \in \mathbb{R}$ such that 
	$\alpha \beta \ne \pm 1$, respectively. Then 
	\begin{equation} \label{2036}
		0<\left|  \alpha \beta-\frac{p_nu_m}{q_nv_m} \right| < \frac{2\beta}{a_{n+1}q_n^2}+\frac{2\alpha}{b_{m+1}v_m^2} ,
	\end{equation}
	where $\{p_n/q_n:n\geq 1\}$ and $\{u_m/v_m:m\geq 1\}$ are the sequences of convergents respectively.
\end{thm}

\begin{proof} By Theorem \ref{thm4}, there exists a sequence of convergents $\{p_n/q_n:n\in \mathbb{N}\}$ such that
	\begin{equation}\label{2008}
		\left|  \alpha-\frac{p_n}{q_n} \right| <\frac{1}{a_{n+1}q_n^2}, 
	\end{equation}
	and the corresponding long form is
	\begin{equation}\label{2009}
		\frac{p_n}{q_n} -\frac{1}{a_{n+1}q_n^2} <\alpha < \frac{p_n}{q_n}+\frac{1}{a_{n+1}q_n^2}. 
	\end{equation}

	Similarly, there exists a sequence of convergents $\{u_m/v_m:m\in \mathbb{N}\}$ such that
	\begin{equation}\label{2028}
		\left|  \beta-\frac{u_m}{v_m} \right| <\frac{1}{b_{m+1} v_m^2},
	\end{equation} 
	and the corresponding long form is
	\begin{equation}\label{2029}
		\frac{u_m}{v_m} -\frac{1}{b_{m+1} v_m^2}  <\beta<\frac{u_m}{v_m} +\frac{1}{b_{m+1} v_m^2}.
	\end{equation}

	The product of the last two long forms returns
	
	\begin{equation}\label{2017}
		\left (\frac{p_n}{q_n}-\frac{1}{a_{n+1}q_n^2}\right )\left (\frac{u_m}{v_m}-\frac{1}{b_{m+1}v_m^2}\right ) <\alpha \beta <\left (\frac{p_n}{q_n}+\frac{1}{a_{n+1}q_n^2}\right )\left (\frac{u_m}{v_m}+\frac{1}{b_{m+1}v_m^2}\right ) .
	\end{equation}
	
	Expanding these expressions produces
	\begin{eqnarray}\label{2018} 
		& & \frac{p_nu_m}{q_nv_m} -\frac{1}{a_{n+1}q_n^2}\frac{u_m}{v_m}-\frac{1}{b_{m+1}v_m^2}\frac{p_n}{q_n} +\frac{1}{a_{n+1}b_{m+1}q_n^2 v_m^2} \nonumber \\
		& <&\alpha \beta \\ &<&
		\frac{p_nu_m}{q_nv_m} +\frac{1}{a_{n+1}q_n^2}\frac{u_m}{v_m}+\frac{1}{b_{m+1}v_m^2}\frac{p_n}{q_n} +\frac{1}{a_{n+1}b_{m+1}q_n^2 v_m^2} . \nonumber
	\end{eqnarray}
	The second order term $1/(a_{n+1}b_{m+1}q_n^2 v_m^2)$ on the left side and right side is absorbed into the larger first order terms on the right side. Thus, rearranging the inequality yield
	\begin{equation}\label{2019} 
		-\frac{1}{a_{n+1}q_n^2}\frac{u_m}{v_m}-\frac{1}{b_{m+1}v_m^2}\frac{p_n}{q_n}  <\alpha \beta -\frac{p_nu_m}{q_nv_m} <\frac{2}{a_{n+1}q_n^2}\frac{u_m}{v_m}+\frac{2}{b_{m+1}v_m^2}\frac{p_n}{q_n} . 
	\end{equation}

	To complete the proof,  rewrite it as a standard inequality 	
	\begin{equation} \label{5036B}
		0<\left|  \alpha \beta-\frac{p_nu_m}{q_nv_m} \right| < \frac{2}{a_{n+1}q_n^2}\frac{u_m}{v_m}+\frac{2}{b_{m+1}v_m^2}\frac{p_n}{q_n} ,
	\end{equation}
	and use the trivial upper bound
	\begin{equation} \label{5100}
		\frac{p_{n}}{q_{n}}  \leq 2 \alpha \qquad \text{  and  } \qquad \frac{u_m}{v_{m}} \leq 2 \beta,
	\end{equation}
	for all large integers $n, m \geq 1$, confer (\ref{2009}) and (\ref{2029}).\\
	
\end{proof}

For distinct irrationals $\alpha, \beta \in (0,1)$, the simpler version
\begin{equation} \label{5036}
	0<\left|  \alpha \beta-\frac{p_nu_m}{q_nv_m} \right| < \frac{2}{a_{n+1}q_n^2}+\frac{2}{b_{m+1}v_m^2}
\end{equation}
can be used to stream line the proof of a result such as Theorem \ref{thm1}.

\begin{thm}[Euler] \label{thm6}  The continued fraction $e=[2,12,1,1,4,1,1,6,1,1,8, \ldots]$ of the natural base has unbounded quotients  
	and the subsequence of convergents \(p_{n}/q_{n}\) satisfies the inequality
	\begin{equation}
		\left| e -\frac{p_{n}}{q_{n}} \right|<\frac{1}{a_{n+1}q_{n}^2}.
	\end{equation}
	
\end{thm}

The quotients have the precise form
\begin{equation}\label{key}
	a_0=2, \qquad a_{3k}=a_{3k-2}=1, \qquad a_{3k-1}=2k,
\end{equation}
for $k \geq 1$. The derivation appears in \cite{OC1970}, \cite[Theorem 2]{LS1995}, \cite[Theorem 3.10]{SJ2005}, \cite{CH2006}, and other.

	%333333333333333333333333333333333333333
	\section{Convergents Correlations} \label{S003}
	The correlation of a pair of convergents $\{p_n/q_n: n\geq 1\}$ and $\{u_m/v_m: n\geq 1\}$ provides information on the distribution of nearly equal values of the continuants $q_n$ and $v_m$.  \\
	
	The regular pattern and unbounded properties of the partial quotients $a_n=a_{3k-1}=2k$ of the continued fraction of $e$, see Theorem \ref{thm6}, are used here to generate a pair of infinite subsequences of rational approximations $\{p_{3k-2}/q_{3k-2}: k \geq 1\}$ and $\{u_{m_k}/v_{m_k}: k \geq 1\}$, for which the product
	\begin{equation} \label{619}
		\frac{p_{3k-2}u_{m_k}}{q_{3k-2}v_{m_k}} \quad \longrightarrow \quad e \pi \qquad \text{ as }k,m_k\longrightarrow \infty .
	\end{equation}
	Furthermore, the values $q_{3k-2} \asymp v_{m_k}$ are sufficiently correlated. The notation $f(x) \asymp g(x)$ is defined by $g(x) \ll f(x) \ll g(x)$.\\
	
	The recursive relations
	\begin{align}
		p_{-1}&=1,           & p_0 &=a_0,              &  p_n&=a_np_{n-1}+p_{n-2}, \nonumber\\
		q_{-1}&=0,           &q_0 &=1,              &  q_n&=a_nq_{n-1}+q_{n-2},
	\end{align}
	for all $n \geq 1$, see  \cite{HW2008}, \cite{NZ1991}, \cite{RH1994}, are used to estimate the rate of growth of the subsequences of continuants $\{q_n: n \geq 1\}$ and $\{v_m: m \geq 1\}$. 
	
	\begin{lem} \label{lem6} 
		Let $e=[a_0,a_1,a_2, \ldots ]$ and  $\pi=[b_0,b_1,b_2, \ldots ]$ be the continued fractions of this pair of irrational numbers. Let $\delta>0$ and $\varepsilon>0$ be a pair of arbitrary small numbers. Then, the followings hold.
		
		\begin{enumerate} [font=\normalfont, label=(\roman*)]
			\item If $b_m=o(m)$, then the exists a pair of subsequences of convergents 
			$p_{3k-2}/q_{3k-2}$ and $u_{m_k}/v_{m_k}$ such that 
			\begin{equation} \label{5162}
				(2k)^{1-\varepsilon}q_{3k-2}\ll v_{m_k} \ll  2 (2k)^{1-\varepsilon}q_{3k-2}  .
			\end{equation}
			\item If $b_m=O(m)$, then the exists a pair of subsequences of convergents 
			$p_{3k-2}/q_{3k-2}$ and $u_{m_k}/v_{m_k}$ such that 
			\begin{equation} \label{5160}
				(2k)^{1-\varepsilon}q_{3k-2}\ll v_{m_k} \ll  2 (2k)^{1-\varepsilon}q_{3k-2}  .
			\end{equation}
			\item If $b_m=O(m^{1+\delta})$, then the exists a pair of subsequences of convergents 
			$p_{3k-2}/q_{3k-2}$ and $u_{m_k}/v_{m_k}$ such that 
			\begin{equation} \label{5164}
				(m_k)^{1-\varepsilon}v_{m_k}\ll q_{n_k} \ll  2 (m_k)^{1-\varepsilon}v_{m_k}  .
			\end{equation}
		\end{enumerate} 
	\end{lem}

	\begin{proof} \textit{Case (i): The partial quotients $b_m=o(m)$ are bounded or unbounded.} Make the change of index $n \equiv 1 \bmod 3  \longrightarrow k=(n+2)/3$ to focus on the subsequence of convergents $p_{3k-2}/q_{3k-2}$ of the number $e$ as $k \to \infty$, see Lemma \ref{lem5} for more details. Observe that
		
		\begin{equation}\label{5188}
			\begin{tabular}{l c l c l}
				$\quad  \vdots$ & & $\quad \quad \quad \vdots$  & & $\quad  \vdots$\\
				$q_{3k-2}$ &=  &$q_{3k-3}+q_{3k-4}$  & = & $q_{3k-2}$ \\ 
				$q_{3k-1}$ & = & $2kq_{3k-2}+q_{3k-3}$ & $\asymp  $ & $ 2^2kq_{3k-2} $  \\ 
				$q_{3k}$ &=  & $ q_{3k-1}+q_{3k-2}$ & $\asymp $ & $2^3kq_{3k-2}$  \\ 
				$q_{3(k+1)-2}$& = & $q_{3(k+1)-3}+q_{3(k+1)-4} $& $\asymp $ &$2^4kq_{3k-2}$  \\ 
				$q_{3(k+1)-1}$& = &$2(k+1)q_{3(k+1)-2}+q_{3(k+1)-3}$  & $\asymp $ & $2^5k(k+1)q_{3k-2}$  \\ 
				$q_{3(k+1)}$& = & $q_{3(k+1)-1}+q_{3(k+1)} $& $\asymp $ &$2^6k(k+1)q_{3k-2}$  \\
				$q_{3(k+2)-2}$& = & $q_{3(k+2)-3}+q_{3(k+2)-4} $& $\asymp $ &$2^7k(k+1)q_{3k-2}$  \\ 
				$q_{3(k+2)-1}$& = &$2(k+2)q_{3(k+1)-2}+q_{3(k+2)-3}$  & $\asymp  $ & $2^8k(k+1)(k+2)q_{3k-2}$  \\ 
				$\quad \vdots$&  & $\quad \quad \quad \vdots$ & & $\quad \vdots$   
			\end{tabular} 
		\end{equation}
		
		This verifies that these numbers has exponential rate of growth in $k$ of the form
		\begin{equation} \label{3300}
			q_{3(k+t)-1}\asymp (4k)^{t+1}q_{3k-2} ,
		\end{equation}
		
		for some $t \geq 0$, as $k \to \infty$. Next, consider the sequence of convergents $u_m/v_m$ of the number $\pi$.  By hypothesis, the partial quotients $b_m=o(m)$ are bounded or unbounded. Furthermore, to simplify the notation, assume that $b_{m_k}\asymp m_{k}^{1-\delta}$ for infinitely many integers $m_k=m_{k_0}, m_{k_1}, m_{k_2}, m_{k_3}, \ldots  \geq 1$. Then, this implies the existence  of a subsequence of convergents $u_{m_k}/v_{m_k}$ such that
		
		\begin{equation}\label{5198}
			\begin{tabular}{l c l c l}
				$\quad \vdots$ & & $\quad \quad \quad \vdots$  & & $\quad \quad \vdots$\\
				$v_{m_k}$ &=  &$b_{m_k}v_{m_k-1}+v_{m_k-2}$  & = & $v_{m_k}$ \\ 
				$v_{m_{k_1}}$ &=  &$b_{m_{k_1}}v_{m_{k_1}-1}+v_{m_{k_1}-2}$  & $\asymp $& $(2^{k_1-k} m_k^{1-\delta}) v_{m_{k}} $ \\
				$v_{m_{k_2}}$ &=  &$b_{m_{k_2}}v_{m_{k_2}-1}+v_{m_{k_2}-2}$  & $\asymp $& $(2^{k_2-k}m_k^{1-\delta})^2 v_{m_{k}}$ \\ 
				$v_{m_{k_3}}$ &=  &$b_{m_{k_3}}v_{m_{k_3}-1}+v_{m_{k_3}-2}$  & $\asymp $& $(2^{k_3-k}m_k^{1-\delta})^3 v_{m_{k}}$ \\
				$\quad \vdots$&  & $\quad \quad \quad \vdots$ & & $\quad \quad \vdots$   
			\end{tabular} 
		\end{equation}
		
		Moreover, the existence of a single value $v_{m_k}>1$, see Tables \ref{t40} and \ref{t80}, such that
		\begin{equation} \label{5172}
			(2k)^{1-\varepsilon}q_{3k-2}\ll v_{m_k} \ll  2 (2k)^{1-\varepsilon}q_{3k-2},
		\end{equation}
		implies the existence of an infinite subsequence of lower bounds 
		\begin{eqnarray} \label{5174}
			v_{m_{k_s}}&=&(2^{k_s-k}m_k^{1-\delta})^s v_{m_k}\\
			&\gg & (2^{k_s-k}m_k^{1-\delta})^s (2k)^{1-\varepsilon}q_{3k-2} \nonumber \\
			&\gg &  (2k)^{1-\varepsilon}q_{3(k+t)-2} \nonumber ,
		\end{eqnarray}
		where 
		\begin{equation}
			q_{3(k+t)-2}= (2^{k_s-k}m_k^{1-\delta})^sq_{3k-2},
		\end{equation}
		and use (\ref{3300}) to identify the relation
		\begin{equation}
			(2^{k_s-k}m_k^{1-\delta})^s=o((4k)^{t+1})
		\end{equation}
		for some $s\geq1$  depending on $t \geq 1$. The corresponding subsequence of upper bounds satisfies 
		\begin{eqnarray} \label{5174}
			v_{m_{k_s}}&=&(2^{k_s-k}m_k^{1-\delta})^s v_{m_k}\\
			&\ll &2 (2^{k_s-k}m_k^{1-\delta})^s (2k)^{1-\varepsilon}q_{3k-2} \nonumber \\
			&\ll & 2 (2k)^{1-\varepsilon}q_{3(k+t)-2} \nonumber .
		\end{eqnarray}
		Combining the last two inequalities yields the required relation
		\begin{equation} \label{6000}
			(2k)^{1-\varepsilon}q_{3(k+t)-2} \ll v_{m_{k_s}} \ll 2 (2k)^{1-\varepsilon}q_{3(k+t)-2}
		\end{equation}
		for some $s,t \geq 1$ as $k \to \infty$.\\
		
		\textit{Case (ii): The the partial quotients $b_m$ are bounded or unbounded, and $b_m=O(m)$.} The proof for this case is similar to Case (i). \\
		
		\textit{Case (ii): The partial quotients $b_m$ are unbounded, and $b_m=O(m^{1+\delta})$}. In this case, (\ref{6000}) can fail, but since the inequality in Theorem \ref{thm5} is symmetric in $q_n$ and $v_m$, the proof is almost the same as Case (i), but the subsequences of convergents are switched to 
		obtain the required relation
		\begin{equation}
			(2^{k_t-k}m_k^{1-\delta})^{t+1} v_{m_{k_t}} \ll q_{n_{k+s}} \ll 2 (2^{k_t-k}m_k^{1-\delta})^{t+1} v_{m_{k_t}}
		\end{equation}
		for some $s,t \geq 1$ as $k \to \infty$. 
	\end{proof} 
	
	The distribution of all the continuants $\{q_n: n \geq 1\}$ associated with a subset of continued fractions of bounded partial quotients is the subject of Zeremba conjecture, see \cite{BK2011} for advanced details. For any continued fraction, the numbers $\{q_n:n \geq 1\}$ have exponential growth 
	\begin{equation}
		q_n=a_nq_{n-1}+q_{n-2}\geq \left ((1+\sqrt{5})/2 \right )^n,
	\end{equation}
	which is very sparse subsequence of integers. The least asymptotic growth occurs for the $(1+\sqrt{5})/2=[1,1,1, \ldots ]$. But the combined subset of continuants for a subset of continued fractions of bounded partial fractions has positive density in the sunset of integers $\N=\{1,2,3, \ldots \}$.

	%444444444444444444444444444444444444444
	%SSSSSSSSSSSSSSSSSSSSSSSSSSSSSSSSSSSSSSSSSSSSSSSSSSSSSSSSSSSSSSSSSSSSSSSSSSSSSSSSSSSSSSSSSSSSSSSSSSSSSSSSSSSSSSSSSSSSSS
	\section{The Main Result} \label{S004}
	\begin{proof} (Theorem \ref{thm1}) Let $e=[a_0,a_1,a_2, \ldots ]$ be the continued fraction of the irrational number $e$. By Theorem \ref{thm4}, there exists a sequence of convergents $\{p_{n}/q_{n}:n\in \mathbb{N}\}$ such that
		\begin{equation}\label{5008}
			\left|  e-\frac{p_{n}}{q_{n}} \right| <\frac{1}{a_{n+1}q_{n}^2}.
		\end{equation}
		Similarly, let $\pi=[b_0,b_1,b_2, \ldots ]$ be the continued fraction of the irrational number $\pi$, and let $\{u_m/v_m:m\in \mathbb{N}\}$ be the sequence of convergents such that
		\begin{equation}\label{5028}
			\left|  \pi-\frac{u_m}{v_m} \right| <\frac{1}{b_{m+1} v_m^2}.
		\end{equation} 
		
		Now, suppose that the product \(e\pi=r/s \in \mathbb{Q}\) is a rational number. Then
		\begin{equation} \label{5046}
			\frac{1}{s}\frac{1}{q_{n}v_m} \leq \left|  e\pi-\frac{p_{n}u_m}{ q_{n}v_m} \right| 
			<\frac{2 \pi}{a_{n+1}q_{n}^2}+\frac{2 e}{b_{m+1}v_m^2}.	
		\end{equation}
		The left side follows from Theorem \ref{thm2}, and the right side follows from Theorem \ref{thm5}.\\
		
		Next, use the subsequence of unbounded partial quotients $a_n=a_{3k-1}$ of the continued fraction of $e$ to generate an infinite subsequence of rational approximations 
		\begin{equation} \label{519}
			\frac{p_{3k-2}u_{m_k}}{q_{3k-2}v_{m_k}} \quad \longrightarrow \quad e \pi \qquad \text{ as }k,m_k\longrightarrow \infty .
		\end{equation}
		
		The subsequence of rational approximations is generated by the following algorithm. \\
		
		Fix an arbitrary small number $\varepsilon>0$.  
		
		\begin{enumerate} \label{3000B}
			\item Input: The integer $n \equiv 1 \bmod 3$.
			\item Let $k=(n+2)/3$.
			\item Let $a_{n+1}=a_{3k-1}=2k$, and fix the convergent $p_{n}/q_{n}=p_{3k-2}/q_{3k-2}$ of $e$, see Theorem \ref{thm6}.
			\item Choose a convergent $u_{m_k}/v_{m_k}$ of $\pi$ in the range 
			\begin{equation} \label{5066}
				(2k)^{1-\varepsilon}q_{3k-2}\leq v_{m_k} \leq  2 (2k)^{1-\varepsilon}q_{3k-2}  .
			\end{equation}
			\item Output:  The quotient $a_{3k-1}=2k$ and the pair of convergents 
			\begin{equation}
				\frac{p_{3k-2}}{q_{3k-2}} \qquad \text{  and  } \qquad  \frac{u_{m_k}}{v_{m_k}}.
			\end{equation}
		\end{enumerate}
		
		Various versions of this algorithm are possible, for example, by modifying the interval in (\ref{5066}). \\
		
		Replacing the subsequence of rational approximations constructed in (\ref{5066}) into (\ref{5046}) yields
		\begin{eqnarray} \label{5076}
			\frac{1}{s}\frac{1}{2(2k)^{1-\varepsilon}q_{3k-2}^2} &\leq &\left|  e \pi -\frac{p_{3k-2}u_{m_k} }{ q_{3k-2} v_{m_k}} \right|  \\
			&\leq& \frac{2 \pi}{2kq_{3k-2}^2} +
			\frac{2e}{(2k)^{2(1-\varepsilon)}b_{m_k+1}q_{3k-1}^2 } \nonumber\\
			&\leq & \frac{\pi}{kq_{3k-2}^2} +\frac{2e}{(2k)^{2(1-\varepsilon)}b_{m_k+1}q_{3k-2}^2 }\nonumber . 
		\end{eqnarray}
		
		Multiplying (\ref{5076}) by $(2k)^{1-\varepsilon}q_{3k-2}^2$ and using $b_{m_k+1} \geq 1$ returns
		\begin{eqnarray} \label{5077}
			\frac{1}{2s} &\leq & \frac{2^{1-\varepsilon} \pi}{k^{\varepsilon}} +\frac{2e}{(2k)^{1-\varepsilon}b_{m_k+1}} \\
			&\leq & \frac{2^{1-\varepsilon} \pi}{k^{\varepsilon}} +\frac{2e}{(2k)^{1-\varepsilon}} \nonumber. 
		\end{eqnarray}
		
		Since \(e\pi=r/s$ is rational constant, and $s \geq 1$, it is clear that the inequality (\ref{5077}) is a contradiction for infinitely many large rational approximations
		\begin{equation}
			\frac{p_{3k-2}}{q_{3k-2}} \frac{u_{m_k}}{v_{m_k}}
		\end{equation}
		as $k,m_k \to \infty$. Ergo, the product \(e\pi\) is not a rational number.
	\end{proof}
	
	The structure of the proof, in equations (\ref{5046}), and (\ref{5076}), is similar to some standard proofs of irrational numbers. Among these well known proofs are the Fourier proof of the irrationality of $e$, see \cite[p.\ 35]{AZ2014}, the proofs for $\zeta(2)$, and $\zeta(3)$ in \cite{BF1979}, et alii.\\
	
	An algorithm and sample of numerical data is compiled in Section \ref{S005} to demonstrate the practicality of this technique.

	%SSSSSSSSSSSSSSSSSSSSSSSSSSSSSSSSSSSSSSSSSSSSSSSSSSSSSSSSSSSSSSSSSSSSSSSSSSSSSSSSSSSSSSSSSSSSSSSSSSSSSSSSSSSSSSSSSSSSSSSSSSSs
	\section{Algorithm And Numerical Data} \label{S005}
	For each fixed pair of index $(n,m)$, the basic product inequality  
	\begin{equation}\label{6029}
		0<\left |e \pi -\frac{p_n}{q_n}\frac{u_m}{v_m} \right |<\frac{2 \pi}{a_nq_n^2}+\frac{2e}{b_mv_m^2}
	\end{equation}
	is used to test each rational approximation. It is quite easy to find the pairs $(n,m)$ to construct a subsequence of rational approximations 
	
	\begin{equation} \label{5049}
		\frac{p_{n}u_{m}}{q_{n}v_{m}}  \quad \rightarrow \quad e \pi
	\end{equation}
	
	as $n, m \to \infty$. The subsequence of rational approximations is generated by the following algorithm. \\
	
	\textbf{Algorithm 1.}  Fix an arbitrary small number $\varepsilon>0$.  
	
	\begin{enumerate} \label{3000}
		\item Input an integer $n \equiv 1 \bmod 3$.
		\item Let $k=(n+2)/3$.
		\item Let $a_{n+1}=a_{3k-1}=2k$, and fix the convergent $p_{n}/q_{n}=p_{3k-2}/q_{3k-2}$ of the natural base $e$.
		\item Choose a convergent $u_{m_k}/v_{m_k}$ of $\pi$ such that 
		\begin{equation} \label{6066}
			(2k)^{1-\varepsilon}q_{3k-2}\leq v_{m_k} \leq  2 (2k)^{1-\varepsilon}q_{3k-2}  .
		\end{equation}
		\item If Step 4 fail, then increment $n \equiv 1 \bmod 3$, and repeat Step 2, the existence is proved in Lemma \ref{lem6}, see also Remark \ref{rem1}.
		\item Output  $a_{3k-1}=2k$, $p_{3k-2}/q_{3k-2}$ and $u_{m_k}/v_{m_k}$	.	
	\end{enumerate}
	
	The parameters for two small but very accurate approximations are listed in the Tables \ref{t40} and \ref{t80}. These examples demonstrate the practicality of the algorithm.

	\begin{table}[h!]
		\begin{center}
			\begin{tabular}{|r|l|}
				\hline 
				$(n,m)=$&(19,10)   \\ [1pt]
				\hline 
				$(a_n,b_m)=$ &$( 12,1)$  \\  [1pt]
				\hline 
				$p_n=$ &13580623  \\  [1pt]
				\hline 
				$q_n=$ &4996032   \\  [1pt]
				\hline 
				$u_m=$ &5419351 \\ [1pt] 
				\hline 
				$v_m=$ &1725033  \\  [1pt]
				\hline
				$\left |e \pi -\frac{p_n}{q_n}\frac{u_m}{v_m} \right |=$ & 0.00000000000012256862192  \\   [1pt]
				\hline
				$\frac{1}{a_nq_n^2}\frac{u_m}{v_m}+\frac{1}{b_m v_m^2}\frac{p_n}{q_n}+\frac{1}{a_{n+1}b_{m+1}q_n^2 v_m^2}=$ & 0.000000000000136378880  \\  [1pt]
				\hline 
			\end{tabular} 
		\end{center}
		\caption{A Rational Approximation Of The Product $e\pi$.} \label{t40}
	\end{table}
	
	\begin{table}[h!]
		\begin{center}
			\begin{tabular}{|r|l|}
				\hline 
				$(n,m)=$&(31,21)   \\ [1pt]
				\hline 
				$(a_n,b_m)=$ &$( 20,2)$  \\  [1pt]
				\hline 
				$p_n=$ &22526049624551  \\  [1pt]
				\hline 
				$q_n=$ &8286870547680   \\  [1pt]
				\hline 
				$u_m=$ &3587785776203 \\ [1pt] 
				\hline 
				$v_m=$ &1142027682075  \\  [1pt]
				\hline
				$\left |e \pi -\frac{p_n}{q_n}\frac{u_m}{v_m} \right |=$ & $8.32849575322710174432272 \times 10^{-25}$  \\   [1pt]
				\hline
				$\frac{1}{a_nq_n^2}\frac{u_m}{v_m}+\frac{1}{b_m v_m^2}\frac{p_n}{q_n}+\frac{1}{a_{n+1}b_{m+1}q_n^2 v_m^2}=$ & $1.04439176914510045201022 \times 10^{-24}$  \\  [1pt]
				\hline 
			\end{tabular} 
		\end{center}
		\caption{A Rational Approximation Of The Product $e\pi$.} \label{t80}
	\end{table}

	\begin{rmk} \label{rem1} {\normalfont 
			Step 5 makes the algorithm independent of the rate of growth of the partial quotients of the number $\pi$. Various versions of this algorithm are possible, for example, by modifying the interval in (\ref{6066}). 
		}
	\end{rmk}

	\currfilename.\\


\newpage
	\begin{thebibliography}{9999}
		\bibitem{AZ2014} Aigner, Martin; Ziegler, Gunter M. \textit{\color{red} Proofs from The Book}. Fifth edition. Springer-Verlag, Berlin, 2014.
		
		\bibitem{BB2004} Berggren, Lennart; Borwein, Jonathan; Borwein, Peter. \textit{\color{red} Pi: a source book}. Third edition. Springer-Verlag, New York, 2004. 
		
		\bibitem{BF1979}  Beukers, F. \textit{\color{red}A note on the irrationality of $\zeta(2)$ and $\zeta(3)$}. Bull. London Math. Soc. 11 (1979), no. 3, 268-272.
		
		\bibitem{BK2011} Jean Bourgain, Alex Kontorovich. \textit{\color{red} On Zaremba's Conjecture}. arXiv:1107.3776.
		
		\bibitem{CH2006} Cohn, Henry. \textit{\color{red} A short proof of the simple continued fraction expansion of e}. Amer. Math. Monthly  113  (2006),  no. 1, 57-62.
		
		\bibitem{EL1744} Leonard Euler, \textit{\color{red} De Fractionibus Continuis}, Desertation, 1744.
		
		\bibitem{HW2008} Hardy, G. H.; Wright, E. M. \textit{\color{red} An introduction to the theory of numbers}. Sixth edition. Oxford University Press, Oxford, 2008. 
		
		\bibitem{LM1997} M. Laczkovich. \textit{\color{red}On Lambert's Proof of the Irrationality of $\pi$}. American Mathematical Monthly, Vol. 104, No. 5 (May, 1997), pp. 439-443. 
		\bibitem{LS1995}  Lang, Serge. \textit{\color{red}Introduction to Diophantine approximations}. Second edition. Springer-Verlag, New York, 1995.
		
		\bibitem{NZ1991} Niven, Ivan; Zuckerman, Herbert S.; Montgomery, Hug h L. \textit{\color{red}An introduction to the theory of numbers.} Fifth edition. John Wiley $\&$ Sons, Inc., New York, 1991. 
		
		\bibitem{NI1947} Niven, Ivan. \textit{\color{red} A simple proof that $\pi$ is irrational}. Bulletin of the American Mathematical Society, 53 (6), p. 509, 1947.
		
		\bibitem{OC1970} Olds, C. D. \textit{\color{red} The Simple Continued Fraction Expansion of e }. Amer. Math. Monthly  77  (1970),  no. 9, 968-974.
		
		\bibitem{PL1953} L. L. Pennisi. \textit{\color{red} Elementary Proof that $e$ is Irrational.} The American Mathematical Monthly, Vol. 60, No. 7 (Aug. - Sep., 1953), p. 474.
		
		\bibitem{RH1994} Rose, H. E. \textit{\color{red} A course in number theory}. Second edition. Oxford Science Publications. The Clarendon Press, Oxford University Press, New York, 1994.
		
		\bibitem{SJ2005}  Steuding, Jorn. \textit{\color{red} Diophantine analysis. Discrete Mathematics and its Applications}. Chapman- Hall/CRC, Boca Raton, FL, 2005.
		
		
	\end{thebibliography}
\end{document}